\documentclass[12pt]{article}
\usepackage{amsmath,amsthm,amscd,amssymb}
\usepackage[mathscr]{eucal}
\usepackage{amssymb}
\usepackage{latexsym}

\theoremstyle{definition}
\newtheorem{Def}{Definition}[section]
\newtheorem{Thm}[Def]{Theorem}
\newtheorem{Prop}[Def]{Proposition}
\newtheorem{Rem}[Def]{Remark}

\newtheorem{Cor}[Def]{Corollary}

\newtheorem{Lem}[Def]{Lemma}

\DeclareMathOperator{\ldt}{ldt}

\numberwithin{equation}{section}
\newcommand{\Q}{\mathbb{Q}}
\newcommand{\R}{\mathbb{R}}
\newcommand{\C}{\mathbb{C}}
\newcommand{\Z}{\mathbb{Z}}
\newcommand{\F}{\mathbb{F}}

\newcommand{\gamt}{\Gamma_{2}}
\newcommand{\gamo}{\Gamma_{1}}
\newcommand{\sym}{\mathrm{Sym}}
\newcommand{\e}{\mathbf{e}}

\newcommand{\prsrs}[1]{[\hspace{-2pt}[ #1 ]\hspace{-2pt}]}
\newcommand{\wx}{\widetilde{X}}
\newcommand{\wsx}{\widetilde{x}}
\newcommand{\wy}{\widetilde{Y}}
\newcommand{\wsy}{\widetilde{y}}
\newcommand{\wm}{\widetilde{M}}
\newcommand{\ww}{\widetilde{W}}
\newcommand{\wf}{\widetilde{f}}
\newcommand{\wg}{\widetilde{g}}
\newcommand{\wsm}{\widetilde{m}}
\newcommand{\trn}[1][1]{{}^t \hspace{-#1pt}}
\newcommand{\mat}[4]{\begin{pmatrix} #1 & #2 \\ #3 & #4 \end{pmatrix}}
\newcommand{\h}{\mathbb{H}}
\newcommand{\zkp}[1][p]{\Z_{(#1)}}
\begin{document}
\title{Sturm bounds for Siegel modular forms of degree 2 and odd weights}
\author{Toshiyuki Kikuta and Sho Takemori}
\maketitle

\begin{abstract}
 We correct the proof of the theorem in the previous paper presented by the first named author, which concerns Sturm bounds for Siegel modular forms of degree $2$ and of even weights modulo a prime number dividing $2\cdot 3$. We give also Sturm bounds for them of odd weights for any prime numbers, and we prove their sharpness. The results cover the case where Fourier coefficients are algebraic numbers.
\end{abstract}
\section{Introduction}
Sturm \cite{sturm1987congruence} studied how many Fourier coefficients we need, when we want to prove that an elliptic modular form vanishes modulo a prime ideal.
Its number is so called ``Sturm bound''. We shall explain it more precisely. For a modular form $f$, let $\Lambda$ be the index
set of the Fourier expansion of $f$. An explicitly given finite subset $S$ of $\Lambda$ is said to be a
\textit{Sturm bound} if vanishing modulo a prime ideal of Fourier coefficients of $f$ at $S$ implies vanishing modulo the prime ideal of all Fourier coefficients of $f$.

Poor-Yuen \cite{poor2015paramodular} studied initially Sturm bounds for Siegel
modular forms of degree $2$ for any prime number $p$. After their study, in \cite{choi2013sturm}, Choi, Choie and the
first named author gave other type bounds with simple descriptions for
them in the case of $p\ge 5$. Moreover, the first named author \cite{kikuta2015remark} attempted to supplement the case of $p\mid 2\cdot 3$. However, there are some gaps in the proof (of Theorem~2.1 in subsection 3.1, \cite{kikuta2015remark}). It seems that its method can only give more larger bounds. Richter-Raum \cite{richter2015sturm} gave some bounds for any $p$ in the case of general degree and any weight. However, their bounds seem not to be sharp except the case of $p\ge 5$ and even weight in degree 2 case. An improvement of their bounds depends on the case of degree $2$.

In this paper, we correct the proof of Theorem 2.1 in \cite{kikuta2015remark} by a new method. Namely we give the sharp Sturm bounds for Siegel modular forms of degree $2$ and even weight in the case of $p=2$, $3$. Moreover we give also the sharp bounds for them of odd weights modulo any prime number $p$. It should be remarked that, their sharpness become important to confirm congruences between two modular forms by numerical experiments, as the weights grow larger. Finally, we remark also that our results cover the case where Fourier coefficients are algebraic numbers.

\section{Statement of the results}
In order to state our results, we fix notation.
For a positive integer $n$,
we define the Siegel modular group $\Gamma_{n}$ of degree $n$ by
\begin{equation*}
  \Gamma_{n} = \left\{\gamma \in \mathrm{GL}_{2n}(\Z)
    \bigm | \trn \gamma J_{n} \gamma = J_{n}
  \right\},
\end{equation*}
where $J_{n} = \mat{0_{n}}{-1_{n}}{1_{n}}{0_{n}}$ and $0_{n}$ (resp. $1_{n}$)
is the zero matrix (resp. the identify matrix) of size $n$.
For a positive integer $N$, we define the principal congruence subgroup $\Gamma^{(n)}(N)$ of
level $N$ by
\begin{equation*}
  \Gamma^{(n)}(N) = \left\{
    \mat{a}{b}{c}{d}  \in \Gamma_{n}\bigm |  \begin{array}{l} a \equiv d \equiv 1_{n} \mod{N}\\ b\equiv c \equiv 0_n \mod{N} \end{array}
  \right\}.
\end{equation*}
Here $a, b, c, d$ are $n \times n$ matrices.
A subgroup $\Gamma \subset \Gamma_{n}$ is said to be a congruence subgroup if
there exists a positive integer $N$ such that $\Gamma^{(n)}(N) \subset \Gamma
\subset \Gamma_{n}$.
For a congruence subgroup $\Gamma$, we say $\Gamma$ is of level $N$ if
$N = \min \left\{m \in \Z_{\ge 1}\bigm | \Gamma^{(n)}(m) \subset \Gamma
\right\}$.
We define the Siegel upper half space $\h_{n}$ of degree $n$ by
\begin{equation*}
  \h_{n}=\left\{x + i y \bigm |
    x \in \sym_{n}(\R), \ y \in \sym_{n}(\R), \ y \text{ is positive definite}
  \right\},
\end{equation*}
where $\sym_{n}(\R)$ is a space of $n \times n$ symmetric matrices with
entries in $\R$.
For a congruence subgroup $\Gamma$ and $k \in \Z_{\ge 0}$,
a $\C$-valued holomorphic function $f$ on
$\h_{n}$ is said to be a (holomorphic) Siegel modular form of degree $n$,
of weight $k$ and of level $\Gamma$
if $f((a Z + b)(c Z + d)^{-1}) = \det \left(cZ + d\right)^{k} f(Z)$
for all $\mat{a}{b}{c}{d} \in \Gamma$.
If $n = 1$, we add the cusp condition.
We denote by $M_{k}(\Gamma)$ the space of Siegel modular forms of weight $k$
and of level $\Gamma$.

Any $f$ in $M_k(\Gamma )$ has a Fourier expansion of the form
\[
f(Z)=\sum_{0\le T\in \frac{1}{N}\Lambda _n}a_f(T)q^T,\quad q^T:=e^{2\pi i\text{tr}(TZ)},
\quad Z\in\mathbb{H}_2,
\]
where $T$ runs over all positive semi-definite elements of $\frac{1}{N}\Lambda _n$, $N$ is the level of $\Gamma $ and
\begin{align*}
\Lambda_n&:=\{ T=(t_{ij})\in\sym_n(\mathbb{Q})\;|\; t_{ii},\;2t_{ij}\in\mathbb{Z}\; \}.
\end{align*}
For simplicity, we write $T=(m,r,n)$ for $T=\begin{pmatrix}m & r/2 \\ r/2 & n \end{pmatrix}\in \frac{1}{N}\Lambda _2$ and also $a_f(m,r,n)$ for $a_f(T)$.

Let $R$ be a subring of $\mathbb{C}$ and $M_{k}(\Gamma )_{R}\subset M_{k}(\Gamma )$ the $R$-module of all modular forms whose Fourier coefficients lie in $R$.

Let $f_1$, $f_2$ be two formal power series of the forms $f_i=\sum_{0\le T\in \frac{1}{N}\Lambda _n}a_{f_i}(T)q^T$ with $a_i\in R$. For an ideal $I$ of $R$, we write
  \begin{equation*}
   f_1 \equiv f_2 \mod{I},
  \end{equation*}
  if and only if $a_{f_1}(T) \equiv a_{f_2}(T) \mod{I}$ for all $T \in \frac{1}{N}\Lambda_{n}$
  with $T \ge 0$.
If $I = (r)$ is a principal ideal, we simply denote $f_{1} \equiv f_{2} \mod{r}$.

Let $K$ be an algebraic number field and ${\mathcal O}={\mathcal O}_K$ the ring of integers in $K$. For a prime ideal $\frak{p}$ in ${\mathcal O}$, we denote by ${\mathcal O}_{\frak{p}}$ the localization of ${\mathcal O}$ at $\frak{p}$. Under these notation, we have
\begin{Thm}
  \label{thm:1}
  \noindent
  Let $k$ be a non-negative integer, $\frak{p}$ an any prime ideal and $f \in
  M_{k}(\Gamma _2)_{{\mathcal O}_{\frak{p}}}$.
  We put
  \begin{equation*}
    b_{k} =
    \begin{cases}
      \left[ \frac{k}{10} \right] & \text{ if } k \text{ is even},\\
      \left[ \frac{k - 5}{10} \right] & \text{ if } k \text{ is odd}.
    \end{cases}
  \end{equation*}
  Here $[\cdot]$ is the Gauss symbol.
  For $\nu \in \Z_{\ge 1}$,
  assume that $a_f(m, r, n) \equiv 0 \bmod{\frak{p}^{\nu}}$ for all $m,\ r, \ n\in \mathbb{Z}$ with
  \begin{align*}
    0\le m, n\le b_{k},
  \end{align*}
  and $4mn - r^{2} \ge 0$, then we have $f \equiv 0 \bmod{\frak{p}^{\nu}}$.
\end{Thm}
\begin{Rem}
  \begin{enumerate}
    \item If $k$ is even and $\mathfrak{p} \nmid 2\cdot 3$, then the statement
    of the theorem was essentially proved by Choi, Choie and the first named author
    \cite{choi2013sturm}.
     \item As mentioned in Introduction, in the case where $\frak{p}\mid 2\cdot 3$ and $k$ is even, the first named author stated the same property in \cite{kikuta2015remark}. However, the proof has some gaps and its method can give only more larger bounds. We give a new proof in subsection \ref{subsec:5.2}.
    \item We note that $M_{k}(\gamt) = \{0\}$ if $k$ is odd and $k < 35$.
    \item Other type bounds also were given in \cite{kikuta2012note}.
  \end{enumerate}
\end{Rem}

By the result of \cite{choi2013sturm} and a similar argument to them, we can prove the following.
\begin{Cor}
  \label{cor:1}
  Let $\Gamma \subset \gamt$ be a congruence subgroup with level $N$, $k \in \Z_{\ge 0}$
  and $f \in M_{k}(\Gamma)_{{\mathcal{O}}_{\mathfrak{p}}}$.
  We put $i = [\gamt : \Gamma]$.
  For $\nu \in \Z_{\ge 1}$, assume that $a_{f}(m, r, n) \equiv 0
  \bmod{\mathfrak{p}^{\nu}}$ for all $m, \ r, \ n
  \in \frac{1}{N}\Z$ with
  \begin{equation*}
    0 \le m, n \le b_{ki}.
  \end{equation*}
  and $4mn - r^{2} \ge 0$, then we have $f \equiv 0 \bmod{\mathfrak{p}^{\nu}}$.
\end{Cor}

In the case of level $1$ (i.e., $N=1$), our bounds are sharp. More precisely, the following theorem holds.
\begin{Thm}
  \label{thm:2}
  Let $k \in \Z_{\ge 0}$ and $p$ be a prime number. We assume $M_{k}(\gamt) \ne 0$.
  Then there exists $f \in M_{k}(\gamt)_{\Z_{(p)}}$ with $f \not \equiv 0 \mod{p}$ such that
  \begin{equation*}
    a_{f}(m, r, n) = 0,\quad \text{ for all } m, \ n \le
    b_{k} - 1.
  \end{equation*}
\end{Thm}

\section{Notation}
\label{sec:3}
For a prime number $p$ and a $\Z_{(p)}$-module $M$, we put
\begin{equation*}
  \widetilde{M} = M \otimes_{\Z_{(p)}}\F_{p}.
\end{equation*}
For an element $x \in M$, we denote by $\widetilde{x}$ the image of $x$ in $\widetilde{M}$.
For a $\Z_{(p)}$-linear map $\varphi : M \rightarrow N$, we denote by
$\widetilde{\varphi}$ the induced map from $\widetilde{M}$ to $\widetilde{N}$
by $\varphi$.
For $n \in \Z_{\ge 1}$, let $\Gamma$ be a congruence subgroup of $\Gamma_{n}$.
We denote $\wm_{k}(\Gamma)_{\Z_{(p)}}$ by $\widetilde{M_{k}(\Gamma)_{\Z_{(p)}}}$.
For a commutative ring $R$ and an $R$-module $M$, we denote by
$\sym^{2}(M) \subset M \otimes_{R} M$ the $R$-module generated by elements
$m \otimes m$ for $m \in M$.
Let $R$ be a $\Z_{(2)}$-algebra and $M$ an $R$-module.
We define an $R$-module $\wedge^{2}(M)$ by
\begin{math}
  \wedge^{2}(M) = \left\{x \in M \bigm | x^{\iota} = - x\right\}.
\end{math}
Here $\iota$ is defined by $\iota(m \otimes n) = n \otimes m$ for
$m, \ n \in M$.
Let $q_{1}, q_{12}, q_{2}$ be variables and
\begin{math}
  S = \left\{q_{1}^{m}q_{12}^{r}q_{2}^{n}\bigm |
    m, n \in \Z_{\ge 0}, r \in \Z
  \right\}
\end{math}
be a set of Laurent monomials.
We define an order of $S$ so that
$q_{1}^{m}q_{12}^{r}q_{2}^{n} \le q_{1}^{m'}q_{12}^{r'}q_{2}^{n'}$
if and only if one of the following conditions holds.
\begin{enumerate}
  \item $m < m'$.
  \item $m = m'$ and $n < n'$.
  \item $m = m'$ and $n = n'$ and $r \le r'$.
\end{enumerate}
Let $K$ be a field and
\begin{math}
  f = \sum_{m, r, n}
  a_{f}(m, r, n)q_{1}^{m}q_{12}^{r}q_{2}^{n} \in
  K[q_{12}, q_{12}^{-1}]\prsrs{q_{1}, q_{2}}
\end{math}
a formal power series. If $f \ne 0$,
let $q_{1}^{m_{0}}q_{12}^{r_{0}}q_{2}^{n_{0}}$ be the minimum
monomial which appears in $f$, that is the minimum monomial
of the set
\begin{math}
  \left\{q_{1}^{m}q_{12}^{r}q_{2}^{n}\bigm |
    a_{f}(m, r, n) \ne 0
  \right\}.
\end{math}
We define the leading term $\ldt(f)$ of $f$ by
$a_{f}(m_{0}, r_{0}, n_{0})q_{1}^{m_{0}}q_{12}^{r_{0}}q_{2}^{n_{0}}$.
We also define the leading term of an element of
$K\prsrs{q_{1}, q_{2}} \setminus
\left\{0\right\}$
by the inclusion $K\prsrs{q_{1}, q_{2}} \subset
K[q_{12}, q_{12}^{-1}]\prsrs{q_{1}, q_{2}}$.
We regard $M_{k}(\gamt)$ as a subspace of $\C[q_{12}, q_{12}^{-1}]
\prsrs{q_{1}, q_{2}}$ by $\sum_{T = (m, r, n) \in \Lambda_{2}}a_{f}(m, r, n)
q^{T} \mapsto \sum_{m, r, n}a_{f}(m, r, n)q_{1}^{m}q_{12}^{r}q_{2}^{n}$.
For $f \in M_{k}(\gamt)$, we denote by $\ldt(f)$ the leading term of
the Fourier expansion of $f$.
For a field $K$, we regard $K\prsrs{q} \otimes_{K} K\prsrs{q}$ as a
subspace of $K\prsrs{q_{1}, q_{2}}$ by $q \otimes 1 \mapsto q_{1}$ and
$1 \otimes q \mapsto q_{2}$.
For a subring $R$ of $\C$ and a subset $S$ of $\C\prsrs{q_{1}, q_{2}}$, we
put
\begin{equation*}
  S_{R}=\left\{f = \sum_{m, n}a_{f}(m, n)q_{1}^{m}q_{2}^{n}\in S\bigm |
  a_{f}(m, n) \in R\right\}.
\end{equation*}
\section{Witt operators}
For the proof of the main results, we use Witt operators.
In this section, we define Witt operators and introduce basic properties of
them.
\subsection{Elliptic modular forms}
Since images of Witt operators can be written by elliptic modular forms,
we introduce some notation for elliptic modular forms.

For $k \in 2 \Z$ with $k \ge 4$, we denote by $e_{k} \in M_{k}(\gamo)$
the Eisenstein series of degree $1$ and weight $k$.
We normalize $e_{k}$ so that the constant term is equal to $1$.
We define Eisenstein series $e_{2}$ of degree $1$ and weight $2$ by
\begin{equation*}
  e_{2}(q) = 1 - 24 \sum_{n = 1}^{\infty}\sigma_{1}(n)q^{n},
\end{equation*}
where $\sigma_{1}(n)$ is the sum of all positive divisors of $n$.
As is well known, $e_{2}$ satisfies the following identify:
\begin{equation*}
  \tau^{-2}e_{2}(-\tau^{-1}) = \frac{12}{2\pi i\tau} + e_{2}(\tau).
\end{equation*}
We put $\Delta = 2^{-6} \cdot 3^{-3} (e_{4}^3 - e_{6}^{2})$. Then $\Delta$
is the Ramanujan's delta function.

For $k \ge 2$,
we define $N_{k}(\gamo)$ as the space of $\C$-valued holomorphic functions $f$
on $\h_{1}$ that satisfies the following three conditions:
\begin{enumerate}
  \item $f(\tau + 1) = f(\tau)$.
  \item There exists $g \in M_{k - 2}(\gamo)$ such that
  \begin{equation*}
    \tau^{-k}f(-\tau^{-1}) = \frac{1}{2 \pi i \tau}g(\tau) + f(\tau) \quad
    \text{ for } \tau \in \h_{1}.
  \end{equation*}
  \item $f$ is holomorphic at the cusp $i \infty$.
\end{enumerate}
Since $f - e_{2}g/12 \in M_{k}(\gamo)$ for above $f$, we have the following
lemma.
\begin{Lem}
  \label{lem:5}
  \begin{equation*}
    N_{k}(\gamo) = M_{k}(\gamo) \oplus e_{2}M_{k - 2}(\gamo).
  \end{equation*}
\end{Lem}
For $M = M_{k}(\gamo)$ or $N_{k}(\gamo)$, we regard
$M$ as a subspace of $\C\prsrs{q}$ via the Fourier expansion.

For $k = 2, 4, 6, 12$, we define elements of
$\sym^{2}\left(N_{k}(\gamo)\right)_{\Z}$ as follows:
\begin{equation}
  \label{eq:5}
  x_{k} = e_{k} \otimes e_{k} ,\text{ for } k = 2, 4, 6, \qquad
  x_{12} = \Delta \otimes \Delta, \qquad
  y_{12} = e_{4}^{3} \otimes \Delta + \Delta \otimes e_{4}^{3}.
\end{equation}
We define $\alpha_{36} \in \wedge^{2}(M_{36}(\gamo))_{\Z}$ by
\begin{equation*}
  \alpha_{36} = x_{12}^{2}(\Delta \otimes e_{4}^{3} - e_{4}^{3} \otimes \Delta).
\end{equation*}

\subsection{Definition of Witt operators}
For $k \in \Z_{\ge 0}$ and $f \in M_{k}(\gamt)$, we consider the following
Taylor expansion
\begin{equation*}
  f(Z)
  = W(f)(\tau_{1}, \tau_{2}) + 2W'(f)(\tau_{1}, \tau_{2})
  \left(2\pi i \tau_{12}\right) + W''(f)(\tau_{1}, \tau_{2})
  \left(2 \pi i\tau_{12}\right)^{2} + O(\tau_{12}^{3}),
\end{equation*}
where $Z = \mat{\tau_{1}}{\tau_{12}}{\tau_{12}}{\tau_{2}} \in \h_{2}$.
We put $q_{1} = \e(\tau_{1}), q_{2} = \e(\tau_{2})$ and
$q_{12} = \e(\tau_{12})$.
By definition,
the following properties hold (see \cite[\S 9]{van2008siegel}).
\begin{enumerate}
  \item $W'(f) = 0$ if $k$ is even and $W(f) = W''(f) = 0$ if $k$ is odd.
  \item $W(f) \in \sym^{2}(M_{k}(\gamo))$ if $k$ is even
  and $W'(f) \in \wedge^{2}(M_{k + 1}(\gamo))$ if $k$ is odd.
  Here we identify $q_{1}$ with $q\otimes 1$ and $q_{2}$ with $1 \otimes q$.
  \item For $f \in M_{k}(\gamt)$ and $g \in
  M_{l}(\gamt)$, we have
  \begin{equation*}
    W(fg) = W(f) W(g), \quad W'(fg) = W'(f) W(g) + W(f) W'(g).
  \end{equation*}
  Assume $k$ and $l$ are both even. Then we have
  \begin{equation}
    \label{eq:8}
    W''(fg) = W''(f)W(g) + W(f)W''(g).
  \end{equation}
  \item For $f = \sum_{m, r, n}a_{f}(m, r, n)q_{1}^{m}q_{12}^{r}q_{2}^{n} \in
  M_{k}(\gamt)$, we have
  \begin{gather*}
    W(f) = \sum_{m, r, n}a_{f}(m, r, n)q_{1}^{m}q_{2}^{n}, \quad
    W'(f) = \frac{1}{2}\sum_{m, r, n}r a_{f}(m, r, n)q_{1}^{m}q_{2}^{n},\\
    W''(f) = \frac{1}{2}\sum_{m, r, n}r^{2}a_{f}(m, r, n)q_{1}^{m}q_{2}^{n}.
  \end{gather*}
\end{enumerate}
Let $k$ be even and $f \in M_{k}(\gamt)$. Then we have
\begin{gather*}
  \tau_{1}^{-k-2}W''(f)(\tau_{1}^{-1}, \tau_{2}) = - \frac{1}{2\pi i}\theta_{2}
  W(f)(\tau_{1}, \tau_{2}) + W''(f)(\tau_{1}, \tau_{2}),\\
  W''(f)(\tau_{1}, \tau_{2}) = W''(f)(\tau_{2}, \tau_{1}).
\end{gather*}
Here $\theta_{2} = \frac{1}{2\pi i}\frac{d}{d\tau_{12}}$.
Therefore by Lemma \ref{lem:5}, we have the following lemma.
\begin{Lem}
  Let $k \in 2\Z_{\ge 0}$ and $f \in M_{k}(\gamt)$. Then we have
  $W''(f) \in \sym^{2}(N_{k + 2}(\gamo))$.
\end{Lem}
Let $R$ be a subring of $\C$.
If $k$ is even and $f \in M_{k}(\gamt)_{R}$, then we have
\begin{equation*}
  W''(f) = \sum_{
    \begin{subarray}{l}
      m, r, n\\
      r > 0
    \end{subarray}}r^{2}a_{f}(m, r, n)q_{1}^{m}q_{12}^{r}q_{2}^{n},
\end{equation*}
since $a_{f}(m, -r, n) = a_{f}(m, r, n)$.
Thus we have $W''(f) \in \sym^{2}(M_{k}(\gamo))_{R}$.
By a similar reason, we have $W'(f) \in M_{k + 1}(\gamt)_{R}$ for
$f \in M_{k}(\gamt)_{R}$ with odd $k$.
For $k \in \Z_{\ge 0}$,
we define $R$-linear maps induced by $W, W'$ and $W''$ as follows.
\begin{align*}
  &W_{R, 2k} : M_{2k}(\gamt)_{R} \rightarrow \sym^{2}(M_{2k}(\gamo))_{R},\\
  &W'_{R, 2k - 1} : M_{2k - 1}(\gamt)_{R} \rightarrow
  \wedge^{2}(M_{2k}(\gamo))_{R},\\
  &W''_{R, 2k}: M_{2k}(\gamt)_{R} \rightarrow
  \sym^{2}(N_{2 k + 2}(\gamo))_{R}.
\end{align*}

\subsection{Igusa's generators and their images}
Let $X_{4}, \ X_{6}, \ X_{10}, \ X_{12}$ and $X_{35}$ be generators of
$\bigoplus_{k \in \Z}M_{k}(\gamt)$ given by Igusa \cite{igusa1962siegel},
\cite{igusa1964siegel}.
Here $X_{4}$ and $X_{6}$ are Siegel-Eisenstein series of weight 4 and 6
respectively. And $X_{10}, X_{12}$ and $X_{35}$ are cusp forms of weight
$10, 12$ and $35$ respectively.
We normalize these modular forms so that
\begin{equation*}
  \ldt(X_{4}) = \ldt(X_{6}) = 1, \quad
  \ldt(X_{10}) = \ldt(X_{12}) = q_{1}q_{12}^{-1}q_{2}^{2}, \quad
  \ldt(X_{35}) = q_{1}^{2}q_{12}^{-1}q_{2}^{3}.
\end{equation*}
Here we note that $a_{X_{35}}(1, r, n) = 0$ for all $n, r \in \Z$,
because a weak Jacobi form of index $1$ and weight $35$ does not exist.
We also introduce $Y_{12} \in M_{12}(\gamt)_{\Z}$ and
$X_{k} \in M_{k}(\gamt)_{\Z}$ for
$k = 16,\allowbreak 18,\allowbreak 24,\allowbreak 28,\allowbreak
30,\allowbreak 36,\allowbreak 40,\allowbreak 42$
and $48$.  Then by Igusa \cite{igusa1979ring},
\begin{equation*}
  \left\{X_{k} \bigm | k = 4, 6, 10,
    12, 16, 18, 24, 28, 30, 36, 40, 42,
    48\right\} \cup \left\{Y_{12}\right\}
\end{equation*}
is a minimal set of generators of $\bigoplus_{k \in 2\Z} M_{k}(\gamt)_{\Z}$
as a $\Z$-algebra and we have $M_{k}(\gamt)_{\Z} = X_{35}M_{k-35}(\gamt)_{\Z}$
for odd $k$.

Igusa \cite{igusa1979ring} computed $W(X_{4}), \cdots, W(X_{48})$ and
$W(Y_{12})$, we introduce some of them.
\begin{align}
  \label{eq:6}
  \nonumber
 &W(X_{4}) = x_{4},\quad
  W(X_{6}) = x_{6}, \quad
  W(X_{10}) = 0, \\
 & W(X_{12}) = 2^{2}\cdot 3 x_{12},\quad W(Y_{12}) = y_{12}, \quad
  W(X_{16})=x_4\cdot x_{12}
\end{align}
and
\begin{equation}
 \label{eq:2}
  W(X_{12i}) = d_{i}x_{12}^{i}, \text{ for } i = 1,2,3, 4.
\end{equation}
Here $x_{4}, x_{6}, x_{12}$ and $y_{12}$ are defined by \eqref{eq:5},
and $d_{i}$ is defined by $12/\mathrm{gcd}(12, i)$.

Images of $W'$ and $W''$ for some of the generators are given as follows.
\begin{Lem}
  \label{lem:8}
  We have
  \begin{equation*}
    W'(X_{35}) = \alpha_{36}.
  \end{equation*}
  and
  \begin{equation*}
    W''(X_{10}) = x_{12}, \quad W''(X_{12i}) = x_{2}x_{12}^{i}, \text{ for }
    i = 1, 2, 3, 4.
  \end{equation*}
\end{Lem}
\begin{proof}
  By $\ldt(X_{35}) = q_{1}^{2}q_{12}^{-1}q_{2}^{3}$ and
  $\wedge^{2}(M_{k}(\gamo)) = (\Delta \otimes e_{4}^{3} - e_{4}^{3} \otimes \Delta)
  \sym^{2}(M_{k - 12}(\gamo))$, we see that $W'(X_{35})$ is a constant
  multiple of $\alpha_{36}$. Since $a_{X_{35}}(2, r, 3) = 0$ if $r
  \neq \pm 1$, we have $W'(X_{35}) = \alpha_{36}$.
  Igusa computed $W''(X_{10})$ and $W''(X_{12})$ (see \cite[Lemma 12]
  {igusa1979ring}). Note that our notation is different from his notation.
  We denote his $W'$ by $W''$.
  By this result, we can compute $W''(X_{12i})$ for $i = 2, 3, 4$.
\end{proof}
\subsection{Kernel of Witt operator modulo a prime}
Let $p$ be a prime number and $k$ even.
We consider the kernel of the Witt operator modulo $p$:
\begin{equation*}
  \ww_{\Z_{(p)}, k} : \wm_{k}(\gamt)_{\Z_{(p)}} \rightarrow
  \sym^{2}(M_{k}(\gamt))_{\Z_{(p)}} \otimes_{\Z_{(p)}}\F_{p}.
\end{equation*}

First we consider the case when $p \ge 5$. This case is easier.

\begin{Lem}
  \label{lem:1}
  Let $p$ be a prime number with $p \ge 5$.
  Then we have
  \begin{equation*}
    \bigoplus_{k \in 2\Z_{\ge 0}}
    \sym^{2}\left(M_{k}(\gamo)\right)_{\Z_{(p)}} =
    \Z_{(p)}[x_{4}, \ x_{6}, \ x_{12}].
  \end{equation*}
\end{Lem}
\begin{proof}
  It is easy to see that
  $\sym^{2}\left(M_{k}(\gamo)_{\Z_{(p)}}\right) = \sym^{2}\left(M_{k}(\gamo)
  \right)_{\Z_{(p)}}$ (see the remark after Theorem 5.12 of \cite{poor2015paramodular}).
  Since $p \ge 5$,
  we have $\bigoplus_{k \in 2\Z_{\ge 0}}M_{k}(\gamo)_{\Z_{(p)}}
  = \Z_{(p)}[e_{4}, \ e_{6}]$ (see \cite{serre1973formes}).
  We note that
  $\bigoplus_{k \in 2\Z_{\ge 0}}\sym^{2}\left(M_{k}(\gamo)_{\Z_{(p)}}\right)$
  is generated by $x_{4}, \ x_{6}$ and $e_{4}^{3}\otimes e_{6}^{2} +
  e_{6}^{2} \otimes e_{4}^{3}$ as an algebra over $\Z_{(p)}$.
  Then the assertion of the lemma follows from the equation
  \begin{equation*}
    2^{12} \cdot 3^{6} x_{12} = x_{4}^{3} + x_{6}^{2} -
    \left(e_{4}^{3}\otimes e_{6}^{2} + e_{6}^{2} \otimes e_{4}^{3} \right).
  \end{equation*}
\end{proof}

The following is a key lemma for the proof of Theorem \ref{thm:1}
for $\mathfrak{p} \nmid 2 \cdot 3$. This lemma was also used in
\cite{choi2013sturm}.
\begin{Lem}
  \label{lem:2}
  Let $p \ge 5$ be a prime number and $k \in 2\Z_{\ge 0}$.
  Then we have
  \begin{equation*}
    \ker\left(\ww_{\zkp, k}\right) = \wx_{10}
    \wm_{k - 10}(\gamt)_{\zkp}.
  \end{equation*}
\end{Lem}
\begin{proof}
  This lemma seems well-known. But for the sake of completeness,
  we give a proof.
  The inclusion $\wx_{10}\wm_{k - 10}(\gamt)_{\zkp} \subset
  \ker\left(\ww_{\zkp, k}\right)$ is obvious, because $W(X_{10}) = 0$.
  Take $f \in M_{k}(\gamt)_{\zkp}$ with $W_{\zkp, k}(f) \equiv 0 \bmod{p}$.
  By \eqref{eq:6} and Lemma \ref{lem:1}, $W_{\zkp, k}$ is surjective.
  Take $g \in M_{k}(\gamt)_{\zkp}$ so that $W_{\zkp, k}(f) = pW_{\zkp, k}(g)$.
  Then by \cite[Corollary 4.2]{nagaoka2000note}, there exists
  $h \in M_{k - 10}(\gamt)_{\zkp}$ such that
  $f - p g= X_{10} h$. This completes the proof.
\end{proof}
\begin{Rem}
  Since $W(X_{12}) = 12 x_{12}$ and $M_{2}(\gamt) = \left\{0\right\}$,
  the assertion of the lemma does not hold if $p = 2, 3$.
\end{Rem}

Next we consider the case where $p = 2, 3$.
We recall the structure of the ring $\bigoplus_{k \in 2\Z}\wm_{k}(\gamt)_{\zkp}$.

\begin{Thm}[Nagaoka \cite{nagaoka2005note}, Theorem 2]
\label{thm:4.7}
  Let $p = 2, 3$. For $f\in \wm_{k}(\gamt)_{\zkp}$, there exists a unique
  polynomial $Q \in \F_{p}[x, y, z]$ such that
  \begin{equation*}
    \widetilde{f} = Q(\wx_{10}, \wy_{12}, \wx_{16}).
  \end{equation*}
\end{Thm}

The above $Q$ for Igusa's generators are given as follows.
\begin{Lem}[Nagaoka \cite{nagaoka2005note}, proof of Lemma 1, Lemma 2]
  \label{lem:10}
  \begin{enumerate}
    \item Suppose $p = 2$, then we have
    \begin{align*}
      & X_{4} \equiv X_{6} \equiv 1 \mod{p},                     &  & X_{12} \equiv X_{10} \mod{p},            \\
      & X_{18} \equiv X_{16} \mod{p},                   &  & X_{24} \equiv X_{10}X_{16} \mod{p},      \\
      & X_{28} \equiv X_{30} \equiv X_{16}^{2} \mod{p}, &  & X_{36} \equiv X_{10} X_{16}^{2} \mod{p}, \\
      & X_{40} \equiv X_{42} \equiv X_{16}^{3} \mod{p}, &  & X_{48} \equiv X_{16}^{4} + X_{10}X_{16}^{3} + X_{10}^{4}Y_{12} \mod{p}, \\ & X_{35}^2\equiv X_{10}^2Y_{12}^2X_{16}^2+X_{10}^6 \mod{p}.
    \end{align*}
    \item Suppose $p = 3$, then we have
    \begin{align*}
      & X_{4} \equiv X_{6} \equiv 1 \mod{p},                        &  & X_{12} \equiv X_{10} \mod{p},                                             \\
      & X_{18} \equiv X_{16} \mod{p},                        &  & X_{24} \equiv X_{10}X_{16} \mod{p},                                       \\
      & X_{28} \equiv X_{30} \equiv X_{16}^{2} \mod{p},             &  & X_{36} \equiv X_{16}^{3} + 2X_{10}^{3} Y_{12} + X_{10}X_{16}^{2} \mod{p}, \\
      & X_{40} \equiv X_{16}^{3} + 2X_{10}^{3} Y_{12} \mod{p},      &  & X_{42} \equiv X_{16}^{3} + X_{10}^{3}Y_{12} \mod{p},                      \\
      & X_{48} \equiv X_{10}X_{16}^{3} + 2X_{10}^{4}Y_{12} \mod{p}, &  &
    \end{align*}
    and
    \begin{multline*}
      X_{35}^2\equiv 2X_{10}X_{16}^4+X_{10}Y_{12}^2X_{16}^3\\
      +2X_{10}^2X_{16}^3+X_{10}^2Y_{12}^2X_{16}^2+ 2X_{10}^3Y_{12}X_{16}^2\\
      +2X_{10}^4Y_{12}^3+X_{10}^4X_{16}^2+2X_{10}^7 \mod{p}.
    \end{multline*}
  \end{enumerate}
\end{Lem}

For later use, we prove the following lemma.
\begin{Lem}
  \label{lem:11}
  Let $p = 2, 3$ and $k \in 2\Z_{\ge 0}$ with $12 \nmid k$.
  Then we have
  \begin{equation*}
    \wm_{k}(\gamt)_{\zkp} \subset \wm_{k + 2}(\gamt)_{\zkp}.
  \end{equation*}
\end{Lem}
\begin{proof}
  Take $f \in M_{k}(\gamt)_{\zkp}$. We show that there exists
  $g \in M_{k + 2}(\gamt)_{\zkp}$ such that $f \equiv g \mod{p}$.
  We may assume $f$ is an isobaric monomial
  of Igusa's generators of even weights, that is $X_{4}, \cdots, X_{48}$ and $Y_{12}$.
  If $f = X_{k}$ with $12 \nmid k$, then by Lemma \ref{lem:10},
  we have $\wf \in \wm_{k + 2}(\gamt)_{\zkp}$.
  In fact, we have $X_{18} \equiv X_{4}X_{16} \mod{p}$,
  $X_{42}\equiv X_{16}X_{28} \bmod{2}$,
  $X_{40} \equiv X_{42} + X_{10}^{3}Y_{12} \mod{3}$
  and $X_{42} \equiv X_{16}X_{28} + X_{10}^{2}X_{12}Y_{12} \mod{3}$.
  If $f$ is an isobaric monomial of weight $k$, then $f$ contains
  some $X_{k}$ with $12 \nmid k$.
  Therefore we have the assertion of the lemma.
\end{proof}

Let $f \in M_{k}(\gamt)_{\zkp}$ with $p = 2, 3$.
As we remarked before, $W(f) \equiv 0 \mod{p}$ does not imply
the existence of $g\in M_{k - 10}(\gamt)_{\zkp}$ such that
$f \equiv X_{10} g \mod{p}$.
Instead of Lemma \ref{lem:2}, we have the following proposition.
\begin{Prop}
  \label{prop:1}
  Let $p = 2, 3$ and $k \in 2\Z_{\ge 0}$.
  \begin{enumerate}
    \item Suppose $12 \nmid k$. Then we have
    \begin{equation*}
       \ker\left(\ww_{\zkp, k}\right) = \wx_{10} \wm_{k - 10}(\gamt)_{\zkp}.
     \end{equation*}
     \item Suppose $k = 12 n$ with $n \in \Z$ and $p = 2$.
     For $0 \le i \le n$ with $4 \nmid i$,
     we put $i = 4s + t$ with $t \in \left\{1, 2, 3\right\}$ and
     $m_{i} = X_{12t}X_{48}^{s}Y_{12}^{n - i}$.
     Then we have
     \begin{equation*}
       \ker\left(\ww_{\zkp, k}\right)
       = \bigoplus_{
         \begin{subarray}{c}
           0 \le i \le n\\
           4 \nmid i
         \end{subarray}}
       \F_{p}\wsm_{i} \oplus \wx_{10} \wm_{k - 10}(\gamt)_{\zkp}.
     \end{equation*}
     \item Suppose $k = 12 n$ with $n \in \Z$ and $p = 3$.
     For $0 \le i \le n$ with $3 \nmid i$,
     we put $i = 3s + t$ with $t \in \left\{1, 2\right\}$.
     and $m_{i} =X_{12t}X_{36}^{s}Y_{12}^{n - i}$.
     Then we have
     \begin{equation*}
       \ker\left(\ww_{\zkp, k}\right)
       = \bigoplus_{
         \begin{subarray}{c}
           0 \le i \le n\\
           3 \nmid i
         \end{subarray}}
       \F_{p}\wsm_{i} \oplus \wx_{10} \wm_{k - 10}(\gamt)_{\zkp}.
     \end{equation*}
   \end{enumerate}
   Moreover, if $f \in M_{k}(\gamt)_{\zkp}$ with $12 \mid k$ and
   \begin{equation*}
     W(f) \equiv W''(f) \equiv 0 \mod{p},
   \end{equation*}
   then there exists $g \in M_{k - 20}(\gamt)_{\zkp}$ such that
   $f \equiv X_{10}^{2} g \mod{p}$.
\end{Prop}
\begin{proof}
  Suppose $12 \nmid k$. Then by \cite[Lemma 13]{igusa1979ring},
  $W_{\Z, k}$ is surjective. Therefore, $W_{\zkp, k}$ is surjective.
  We can prove
  \begin{math}
    \ker\left(\ww_{\zkp, k}\right) = \wx_{10} \wm_{k - 10}(\gamt)_{\zkp}
  \end{math}
  by a similar argument to the proof of Lemma \ref{lem:2}.
  Next assume $k = 12n$ with $n \in \Z$.
  For simplicity, we assume $p = 2$. We can prove the case when $p = 3$
  in a similar way.
  Take $f \in M_{k}(\gamt)_{\zkp}$ with $W(f) \equiv 0 \mod{p}$.
  Put $d_{i} = 12/\mathrm{gcd}(12, i)$.
  By \cite[Lemma 13]{igusa1979ring}, there exist
  $a_{i, j}, b_{i}, c_{i} \in \zkp$ such that
  \begin{equation*}
    W(f) = \sum_{0 \le i\le j < n}a_{i, j}x_{4}^{3(n - j)}x_{12}^{i} y_{12}^{j
    - i}
    + \sum_{
      \begin{subarray}{c}
        0 \le i \le n\\
        4 \mid i
      \end{subarray}}
    b_{i}x_{12}^{i}y_{12}^{n - i}
    + \sum_{
      \begin{subarray}{c}
        0 \le i \le n\\
        4 \nmid i
      \end{subarray}}
    c_{i}d_{i}x_{12}^{i}y_{12}^{n - i}.
  \end{equation*}
  By $x_{4} \equiv 1 \mod{p}$ and $W(f) \equiv 0 \mod{p}$, we have
  $a_{i, j} \equiv b_{i} \equiv 0 \mod{p}$ for all $i, j$.
  Here we note that $\wsx_{12}$ and $\wsy_{12}$ are algebraically independent
  over $\F_{p}$. This is because $\ldt(x_{12}^{i}y_{12}^{j}) =
  q_{1}^{i}q_{2}^{i + j}$.
  By \cite[Lemma 13]{igusa1979ring}, there exists $f' \in M_{k}(\gamt)_{\zkp}$
  such that
  \begin{equation*}
    W(f') = \sum_{0 \le i\le j < n}\frac{a_{i, j}}{p}
    x_{4}^{3(n - j)}x_{12}^{i} y_{12}^{j - i}
    + \sum_{
      \begin{subarray}{c}
        0 \le i \le n\\
        4 \mid i
      \end{subarray}}
    \frac{b_{i}}{p}x_{12}^{i}y_{12}^{n - i}.
  \end{equation*}
  By \eqref{eq:2},
  there exists $u_{i} \in \Z_{(p)}^{\times}$ such that
  $W(m_{i}) = u_{i}d_{i} x_{12}^{i}y_{12}^{n - i}$.
  Therefore, there exist $a_{i} \in \Z_{(p)}$ such that
  $W(f - pf' - \sum_{
    \begin{subarray}{l}
      0 \le i \le n\\
      4 \nmid i
    \end{subarray}
  }a_{i}m_{i}
  ) = 0$. By \cite[Corollary 4.2]{nagaoka2000note}, there exists $g \in M_{k - 10}(\gamt)_{\zkp}$
  such that
  \begin{math}
    \wf = \sum_{i}\widetilde{a}_{i}\wsm_{i} + \wx_{10}\widetilde{g}.
  \end{math}
  Thus we have
  \begin{equation}
    \label{eq:7}
       \ker\left(\ww_{\zkp, k}\right)
       = \sum_{
         \begin{subarray}{c}
           0 \le i \le n\\
           4 \nmid i
         \end{subarray}}
       \F_{p}\wsm_{i} + \wx_{10} \wm_{k - 10}(\gamt)_{\zkp}.
  \end{equation}
  We show that the sum \eqref{eq:7} is a direct sum.
  Let $a_{i}\in \Z_{(p)}$ for $0 \le i \le n$ with $4 \nmid i$
  and $g \in M_{k-10}(\gamt)_{\zkp}$. We put $f = \sum_{i}a_{i}m_{i} +
  X_{10}g$.
  By \eqref{eq:8}, we have
  \begin{equation}
    \label{eq:9}
    W''(m_{i}) \equiv W''(X_{12t})W(X_{48}^{s}Y_{12}^{n-i}) \equiv
    x_{12}^{i}y_{12}^{n-i} \mod{p}.
  \end{equation}
  Here we use $W(X_{12t}) \equiv 0 \mod{p}$ for $t = 1, 2, 3$
  and $x_{2} \equiv 1 \mod{p}$.
  By Igusa's computation, images of 14 generators $X_{4}, \cdots, X_{48}$
  by $W$
  can be written as $\Z$-coefficient polynomials of $x_{4}, x_{6}, x_{12}$ and $y_{12}$.
  By Lemma \ref{lem:8}, we have $W''(X_{10}) = x_{12}$.
  Thus there exist $\alpha_{a, b, c, d} \in \Z_{(p)}$ such that
  \begin{equation*}
    W''(X_{10}g)
    = x_{12}W(g)
    = \sum_{a, b, c, d}\alpha_{a, b, c, d} x_{4}^{a}x_{6}^{b}x_{12}^{c}y_{12}^{d},
  \end{equation*}
  where summation index runs over $\left\{(a, b, c, d) \in \Z_{\ge 0}^{4}
    \bigm |
    4a + 6b + 12 c + 12 d = k + 2\right\}$.
  We assume $\ww''_{\zkp, k}(\wf) = \ww''_{\zkp, k}
  (\sum_{i}\widetilde{a}_{i}\wsm_{i} + \wx_{10}\widetilde{g}) = 0$.
  Then by \eqref{eq:9} and $x_{4} \equiv x_{6} \equiv 1 \mod{p}$, we have
  \begin{equation*}
    \sum_{i}\widetilde{a}_{i}\wsx_{12}^{i}\wsy_{12}^{n - i} +
    \sum_{a, b, c, d}\widetilde{\alpha}_{a, b, c, d}\wsx_{12}^{c}\wsy_{12}^{d} = 0.
  \end{equation*}
  Since $4a + 6b = 0$ or $4a + 6b \ge 4$, the isobaric degree of
  $\wsx_{12}^{c}\wsy_{12}^{d}$ is not equal to $k$.  Therefore we have
  $\widetilde{a}_{i} = 0$ for all $i$.  This shows that the sum \eqref{eq:7} is a direct
  sum. This also shows that if $f \in M_{k}(\gamt)_{\zkp}$ with $12 \mid k$
  satisfies $W(f) \equiv W''(f) \equiv 0 \mod{p}$, then there exists
  $h \in M_{k - 10}(\gamt)_{\zkp}$ such that $f \equiv X_{10}h\mod{p}$. By
  $W''(f) \equiv 0 \mod{p}$, we have $W(h) \equiv 0 \mod{p}$.
  Since $12 \nmid k - 10$, there exists $h' \in M_{k - 20}(\gamt)_{\zkp}$ such
  that $h \equiv X_{10}h' \mod{p}$.
  Therefore we have $f \equiv X_{10}^{2}h' \mod{p}$.
  This completes the proof.
\end{proof}
\begin{Cor}
  \label{cor:3}
  Let $p = 2, 3$ and $f \in M_{k}(\gamt)_{\zkp}$ with $12 \mid k$.
  If $W(f) \equiv 0 \mod{p}$, then there exists $g \in M_{k - 8}(\gamt)_{\zkp}$
  such that $f \equiv X_{10}g \mod{p}$.
\end{Cor}
\begin{proof}
  By Lemma \ref{lem:10}, the statement for $f = m_{i}$ is true for all $i$.
  Then by Proposition \ref{prop:1}, we have
  $f \equiv X_{10}(g + h) \mod{p}$ with $g \in M_{k - 8}(\gamt)_{\zkp}$ and
  $h \in M_{k - 10}(\gamt)_{\zkp}$.
  By Lemma \ref{lem:11}, we have our assertion.
\end{proof}

\section{Proof of the main results}
In this section, we give proofs of Theorem \ref{thm:1},
Corollary \ref{cor:1} and Theorem \ref{thm:2}.

We have
\begin{math}
  \wm_{k}(\gamt)_{\mathcal{O}_{\mathfrak{p}}}
  =
  \wm_{k}(\gamt)_{\Z_{(p)}}\otimes_{\F_{p}}
  \mathcal{O}_{\mathfrak{p}}/\mathfrak{p}.
\end{math}
Therefore Theorem \ref{thm:1} is reduced to the case of
$\mathcal{O}_{\mathfrak{p}} = \Z_{(p)}$, where $p$ is a prime number.
We also note that the statement of Theorem \ref{thm:1} for $\nu \ge 2$
is reduced to the case of $\nu = 1$ by repeatedly using the result.
This method was used in \cite{rasmussen2009higher}.

As we remarked before, the statement of Theorem \ref{thm:1}
was proved in \cite{choi2013sturm} if $k$ is even and $p \ge 5$.
Thus in this section, we assume
$k \equiv 0 \mod{2},\ p = 2, 3$
or $k \equiv 1 \mod{2}$.

First, we introduce the following notation, which is similar to
mod $p$ diagonal vanishing order defined by
Richter and Raum \cite{richter2015sturm}.
Let $\wf$ be a $\F_{p}$-coefficients formal power series as follows;
\begin{equation*}
  \wf = \sum_{
    \begin{subarray}{c}
      m, r, n \in \Q\\
      m, n, 4mn - r^{2}\ge 0
    \end{subarray}}
    \widetilde{a}_{f}(m, r, n)q_{1}^{m}q_{12}^{r}q_{2}^{n}
    \in \bigcup_{N \in \Z_{\ge 1}}\F_{p}[q_{12}^{1/N}, q_{12}^{-1/N}]
    \prsrs{q_{1}^{1/N}, q_{2}^{1/N}}.
\end{equation*}
We define $v_{p}(\wf)$ by
\begin{equation*}
  v_{p}(\wf) = \sup
  \left\{
    A \in \R\bigm |
    \begin{array}{l}
      \widetilde{a}_{f}(m, r, n) = 0, \\
      \text{for all } m, r, n \in \Q \text{ with }
      0 \le m, n < A
    \end{array}
  \right\}.
\end{equation*}
By definition, we have
\begin{equation}
  \label{eq:11}
  v_{p}(\wf \wg) \ge \max\left\{
    v_{p}(\wf), v_{p}(\wg)
  \right\},
\end{equation}
for $\wf, \wg \in \bigcup_{N \in \Z_{\ge 1}}\F_{p}[q_{12}^{1/N}, q_{12}^{-1/N}]
\prsrs{q_{1}^{1/N}, q_{2}^{1/N}}$.
We note that
$v_{p}(\wf) > A$ is equivalent to
$\widetilde{a}_{f}(m, r, n) = 0$ for all $m, n \le A$,
where $A \in \R$.

For the proof of Theorem \ref{thm:1}, we introduce the following three lemmas.

\begin{Lem}
  \label{lem:3}
  Let $p$ be a prime number and
  $\wf \in \wm_{k}(\gamt)_{\zkp}$ with $k \in \Z_{\ge 0}$.
  Then we have $v_{p}(\wx_{10} \wf) = v_{p}(\wf) + 1$ and $v_p(\wx_{35}\wf)\ge v_p(\wf)+2$.
\end{Lem}
\begin{proof}
  We regard $\wx_{10}$ and $\wx_{35}$ as images
  in the ring of formal power series
  $\F_{p}(q_{12})\prsrs{q_{1}, q_{2}}$.
  By the Borcherds product of $X_{10}$ (cf. \cite{gritsenko1997siegel}),
  we have $\widetilde{X}_{10} = q_{1} q_{2} u$ where $u$ is a unit in
  $\F_{p}(q_{12})\prsrs{q_{1}, q_{2}}$.
  Similarly, we have $\widetilde{X}_{35} = q_{1}^2 q_{2}^2 (q_1-q_2)v$ for
  some unit $v$ in $\F_{p}(q_{12})\prsrs{q_{1}, q_{2}}$ (cf. \cite{gritsenko1996igusa}).
  The assertion of the lemma follows from these facts.
\end{proof}
\begin{Rem}
  It is not easy to give an upper bound for $v_{p}(\wx_{35}\wf) - v_{p}(\wf)$
  because of the factor $q_{1} - q_{2}$ in the Borcherds product of
  $X_{35}$.
\end{Rem}

\begin{Lem}
  \label{lem:12}
  Let $p$ be a prime number and
  \begin{equation*}
    f = \sum_{m, n \ge 0}a_{f}(m, n)q_{1}^{m}q_{2}^{n}
    \in \left( M_{k}(\gamo) \otimes M_{k}(\gamo)
  \right)_{\zkp}.
  \end{equation*}
  If $a_{f}(m, n)\equiv 0 \mod{p}$ for all $m, n \le [k/12]$,
  then $f \equiv 0 \mod{p}$.
  In particular, for $g \in M_{k}(\gamt)_{\zkp}$,
  we have $W(g) \equiv 0 \mod{p}$ if $v_{p}(\widetilde{g})
  > [k/12]$ and
  $W'(g) \equiv 0 \mod{p}$ if $v_{p}(\widetilde{g}) > [(k + 1)/12]$.
\end{Lem}
\begin{proof}
  By the original Sturm's theorem \cite{sturm1987congruence}, the map
  \begin{equation*}
    \wm_{k}(\gamo)_{\zkp} \hookrightarrow \F_{p}\prsrs{q}/(q^{[k/12] + 1})
  \end{equation*}
  is injective.
  Therefore we have the following injective map
  \begin{multline*}
    \sym^{2}(M_{k}(\gamo))_{\zkp}\otimes_{\zkp}\F_{p} =
    \sym^{2}(\wm_{k}(\gamo)_{\zkp})\\
    \hookrightarrow
    F_{p}\prsrs{q}/(q^{[k/12] + 1})\otimes_{\F_{p}}
    F_{p}\prsrs{q}/(q^{[k/12] + 1}).
  \end{multline*}
  Here we note that
  $\sym^{2}(M_{k}(\gamo)_{\zkp}) = \sym^{2}(M_{k}(\gamo))_{\zkp}$,
  as we remarked in the proof of Lemma \ref{lem:1}.
  Since the image of $\wf$ by this map vanishes,
  we have $\wf = 0$.
\end{proof}

\begin{Lem}
  \label{lem:4}
  We define $f_{k} \in M_{k}(\gamt)_{\Z}$ for $k = 35, 39, 41, 43$ and $47$ as
  follows.
  \begin{gather*}
    f_{35} = X_{35}, \quad f_{39} = X_{4} X_{35}, \quad f_{41} = X_{6}X_{35},
    \quad
    f_{43} = X_{4}^{2}X_{35}, \quad f_{47} = X_{12} X_{35}.
  \end{gather*}
  Then $\ldt(f_{k}) = q_{1}^{2}q_{12}^{-1}q_{2}^{3}$
  for $k = 35, 39, 41, 43$
  and $\ldt(f_{47}) = q_{1}^{3}q_{12}^{-2}q_{2}^{4}$.
\end{Lem}
\begin{proof}
  This follows from $\ldt(X_{4}) = \ldt(X_{6}) = 1$,
  $\ldt(X_{12}) = q_{1}q_{12}^{-1}q_{2}$ and
  $\ldt(X_{35}) = q_{1}^{2}q_{12}^{-1}q_{2}^{3}$.
\end{proof}

\subsection{Proof of Theorem ~\ref{thm:1} for $p = 2, 3$ and even $k$}
\label{subsec:5.2}
Let $p = 2, 3$, $k \in 2\Z_{\ge 0}$ and $f \in M_{k}(\gamt)_{\zkp}$.
We assume
\begin{equation}
  \label{eq:1}
  v_{p}(\wf) > b_{k},
\end{equation}
where $b_{k}$ is given in Theorem \ref{thm:1}.
We prove the statement of Theorem \ref{thm:1} by the induction on $k$.
First, we assume $k < 10$.
Then the statement is true because $M_{k}(\gamt)$ for $k = 4, 6, 8$ is one-dimensional
and $\ldt(X_{4}) = \ldt(X_{6}) = \ldt(X_{4}^{2}) = 1$.

Next, we assume $k \ge 10$ and the statement is true if the weight is
strictly less than $k$.
By \eqref{eq:1} and Lemma \ref{lem:12}, we have $W(f) \equiv 0 \mod{p}$.
If $12 \nmid k$, then by
Proposition \ref{prop:1}, there exists
$g \in M_{k - 10}(\gamt)_{\zkp}$ such that $f \equiv X_{10} g \mod{p}$.
By \eqref{eq:1} and Lemma \ref{lem:3}, we have
$v_{p}(\widetilde{g}) > b_{k - 10}$.
By the induction hypothesis, we have $g \equiv 0 \mod{p}$.
Thus we have the assertion of Theorem \ref{thm:1} in this case.
Next we assume $12 \mid k$. Then by Corollary \ref{cor:3}, there exists $g \in M_{k - 8}(\gamt)_{\zkp}$
such that $f \equiv X_{10}g \mod{p}$.
Since $b_{k - 10} \ge [(k - 8) / 12]$ for $k \ge 10$,
we have $W(g) \equiv 0 \mod{p}$ by
\eqref{eq:1}, Lemma \ref{lem:3} and Lemma \ref{lem:12}.
Therefore $W''(f) \equiv x_{12}W(g) \equiv 0 \mod{p}$.
By Proposition \ref{prop:1}, there exists $h \in M_{k - 20}(\gamt)_{\zkp}$ such that
$f \equiv X_{10}^{2}h \mod{p}$.
Since $v_{p}(\widetilde{h}) > b_{k - 20}$,
we have $h \equiv 0 \mod{p}$ by the induction hypothesis.
Thus we have $f\equiv 0 \mod{p}$. This completes the proof.\qed

\subsection{Proof of Theorem~\ref{thm:1} for the case $p \nmid 2\cdot 3$
and odd $k$}
Let $p$ be a prime number with $p \ge 5$ and $f \in M_{k}(\gamt)_{\Z_{(p)}}$
with $k$ odd.
We assume
\begin{equation}
  v_{p}(\wf) > b_{k}.
  \label{eq:3}
\end{equation}

We prove the theorem by the induction on $k$.
Note that $M_{k}(\gamt) = \left\{0\right\}$ if $k$ is odd and $k < 35$ or $k
= 37$.
First assume that $ 0 \le k - 35 < 10$ with $k \ne 37$.
Then $M_{k}(\gamt)$ is one-dimensional and spanned by $f_{k}$ given in Lemma
\ref{lem:4}.
By Lemma \ref{lem:4}, the assertion of the theorem holds if $k - 35 < 10$.

Next we assume $k -35 \ge 10$ and the assertion of the theorem holds if
the weight is strictly less than $k$.
By Igusa \cite{igusa1979ring}, there exists $g \in M_{k - 35}(\gamt)_{\zkp}$
such that $f = X_{35} g$.
By Lemma \ref{lem:8}, we have
\begin{align}
  W'(f) = W'(X_{35}) W(g) = \alpha_{36} W(g). \label{eq:4}
\end{align}
By $[(k + 1)/12] \le b_{k}$ and Lemma \ref{lem:12}, we have $W'(f)
\equiv 0 \bmod{p}$.
Therefore,
we have $W(g) \equiv 0 \bmod{p}$ by \eqref{eq:4}. Then by Lemma \ref{lem:2}, there exists
$g' \in M_{k- 45}(\gamt)_{\Z_{(p)}}$ such that
$g \equiv X_{10} g' \bmod{p}$.
We put $f' = X_{35}g'$. Then we have $f \equiv X_{10}f' \bmod{p}$.
By \eqref{eq:3} and Lemma \ref{lem:3}, we have
$v_{p}(\wf') > b_{k - 10}$.
By the induction hypothesis, we have $f' \equiv 0 \bmod{p}$.
Thus $f \equiv 0 \bmod{p}$. This completes the proof. \qed

\subsection{Proof of Theorem~\ref{thm:1} for $p = 2, 3$ and odd $k$}
In this subsection, we assume $p = 2, 3$ and $k$ is odd.
Since the case when $k = 48 + 35 = 83$ is special in our proof, we prove the
following two lemmas.
\begin{Lem}
  \label{lem:6}
  Let $\wf \in \wm_{48}(\gamt)_{\zkp}$ with $\wf \ne 0$ and
  $\ldt(\wf) = \alpha q_{1}^{a}q_{12}^{b}q_{2}^{c}$ be the leading term of
  $\wf$. Here $\alpha \in \F_{p}^{\times}$.
  Assume $\ww_{\zkp, 48}(\wf) = 0$.
  Then we have $a \le 4$ and $c \le 4$.
\end{Lem}
\begin{proof}
  By Proposition \ref{prop:1}, we have
  \begin{equation*}
    \ker(\ww_{\zkp, 48}) = \bigoplus_{i}\F_{p}\wsm_{i}
    \oplus \wx_{10} \wm_{38}(\gamt)_{\zkp}.
  \end{equation*}
  Here $i = 1, 2, 3$ if $p = 2$ and $i = 1, 2, 4$ if $p = 3$.
  For $\widetilde{g} \in \wm_{48}(\gamt)_{\zkp}$, let $Q_{g}=
  \sum_{a, b,
    c}\gamma_{a, b, c}x^{a}y^{b}z^{c}$ be
  a $\F_{p}$-coefficients polynomial
  such that $\widetilde{g} = Q_{g}(\wx_{10}, \wy_{12}, \wx_{16})$ as in
  Theorem \ref{thm:4.7}.
  Since
  \begin{equation}
    \label{eq:10}
    \ldt\left(\wx_{10}^{a}\wy_{12}^{b}\wx_{16}^{c}\right) = q_{1}^{a + c}
    q_{12}^{-a}q_{2}^{a + b + c},
  \end{equation}
  there exists a unique monomial
  $\wx_{10}^{a_{0}}\wy_{12}^{b_{0}}\wx_{16}^{c_{0}}$ with $\gamma_{a_{0}, b_{0}, c_{0}} \ne 0$
  such that
  $\ldt(\widetilde{g}) = \ldt(\gamma_{a_0, b_0, c_0}\wx_{10}^{a_{0}}\wy_{12}^{b_{0}}\wx_{16}^{c_{0}})$.
  We put $\phi(\widetilde{g}) = \wx_{10}^{a_{0}} \wy_{12}^{b_{0}}\wx_{16}^{c_{0}}$.
  We define a set $S'$ by
  \begin{gather*}
    \left\{
      1,
      \wx_{16},
      \wy_{12},
      \wx_{10},
      \wx_{16}^{2},
      \wy_{12} \wx_{16},
      \wy_{12}^{2},
      \wx_{10} \wx_{16},
      \wx_{10} \wy_{12},
      \wx_{10}^{2},\right.\\
    \left.
      \wx_{10} \wx_{16}^{2},
      \wx_{10} \wy_{12} \wx_{16},
      \wx_{10} \wy_{12}^{2},
      \wx_{10}^{2} \wx_{16},
      \wx_{10}^{2} \wy_{12},
      \wx_{10}^{3}
    \right\}.
  \end{gather*}
  Then $S'$ forms a basis of $\wm_{38}(\gamt)_{\zkp}$.
  This follows from $\dim_{\F_{p}}(\wm_{38}(\gamt)_{\zkp}) =
  \dim_{\C}M_{38}(\gamt) = 16$ and Lemma \ref{lem:10}.
  We put $S = \left\{\wx_{10} a \bigm | a \in S'\right\}$.
  We define a set $T$ by
  \begin{equation*}
    T =
    \begin{cases}
      \left\{\wsm_{1}, \wsm_{2}, \wsm_{3}\right\} & \text{ if } p = 2,\\
      \left\{\wsm_{1}, \wsm_{2}, \wsm_{4}\right\} & \text{ if } p = 3.
    \end{cases}
  \end{equation*}
  Then $S \cup T$ forms a basis of $\ker(\ww_{\zkp, 48})$.
  We have $\phi(s) = s$ except when $p = 3$ and $s = m_{4}$ for $s \in S\cup T$.
  If $p = 3$, we have $\phi(\wsm_{4}) = \wx_{10}\wy_{12}\wx_{16}^{2}$.
  Thus we see that
  $\phi$ on $S \cup T$ is injective.
  Therefore if $\wf \in \ker(\ww_{\zkp, 48})$ with $\wf \ne 0$, then
  there exists a unique $s \in S \cup T$ such that
  $\ldt(\wf) = \alpha \ldt(s)$ with $\alpha \ne 0$.
  Note that degrees of
  monomials $\left\{\phi(s) \bigm | s \in S\cup T\right\}$
  are less than or equal to $4$.
  Then by \eqref{eq:10}, we have the assertion of the lemma.
\end{proof}

\begin{Lem}
  \label{lem:7}
  Let $k = 83$, $\wf = \wx_{35}
  \widetilde{g}\in \wm_{k}(\gamt)_{\zkp}$ with
  $g \in \wm_{k - 35}(\gamt)_{\zkp}$ and
  $\ww_{\zkp, k - 35}(\widetilde{g}) = 0$.
  Assume $v_{p}(\wf) > b_{k} = 7$. Then we have
  $\wf = 0$.
\end{Lem}
\begin{proof}
  Assume $\wf \ne 0$.
  We put $\ldt(\widetilde{g}) = \alpha q_{1}^{a}q_{12}^{b}q_{2}^{c}$, where
  $\alpha \in \F_{p}^{\times}$.
  Then by Lemma \ref{lem:6}, we have $a, c \le 4$.
  Since $\ldt(\wx_{35}) = q_{1}^{2}q_{12}^{-1}q_{2}^{3}$, we have
  $\ldt(\wf) = \alpha q_{1}^{a + 2}q_{12}^{b-1}q_{2}^{c + 3}$.
  By $a + 2 \le 6$ and $c + 3 \le 7$, we have $v_{p}(\wf) \le 7$.
\end{proof}

Let $k$ be odd and $f \in M_{k}(\gamt)_{\zkp}$. Assume that
\begin{align}
  \label{k4}
  v_p(\wf) > b_{k}.
\end{align}
If $k < 45$, then the assertion follows from Lemma \ref{lem:4}.
Hence we suppose that $k \ge 45$. To apply an induction on $k$, suppose that the assertion is true for any weight strictly smaller than $k$.

We take $g\in M_{k-35}(\Gamma _2)_{\mathbb{Z}_{(p)}}$ such that $f=gX_{35}$.
By \eqref{k4}, \eqref{eq:4} and Lemma \ref{lem:12}, we have
$W(g) \equiv 0 \mod{p}$.
Now we separate into four cases:\\
(1) If $k \not \equiv 11 \mod{12}$, then there exists $g'\in M_{k-45}(\Gamma _2)_{\mathbb{Z}_{(p)}}$ such that
$g \equiv X_{10}g' \mod{p}$, by Proposition \ref{prop:1}. Then $f=X_{35}g=X_{35}X_{10}g'$. If we put $f':=X_{10}g'\in M_{k-10}(\Gamma _2)_{\mathbb{Z}_{(p)}}$, then
\[
  b_{k} < v_p(\wf)=v_{p}(\wx_{10}\wf ')=1+v_p(\wf ').
\]
This implies $v_p(\wf ')> b_{k - 10}$. By the induction hypothesis, we get $f'\equiv 0$ mod $p$.
Therefore we have $f\equiv 0 \mod{p}$.

\noindent
(2) If $k \equiv 11 \mod{12}$ and $k \equiv 1, 5, 7, 9 \mod{10}$, then we have
$b_{k}= b_{k - 8} + 1$. By Corollary \ref{cor:3}, there exists
$g' \in M_{k - 43}(\gamt)_{\zkp}$ such that $g \equiv X_{10}g' \mod{p}$.
Put $f' = X_{35}g' \in M_{k-8}(\gamt)_{\zkp}$.
Then we have $v_{p}(\wf') = v_{p}(\wf) - 1 > b_{k - 8}$. By the induction
hypothesis, we have $f' \equiv 0 \mod{p}$.
Therefore we have $f \equiv 0 \mod{p}$.

\noindent
(3) If $k \equiv 11 \mod{12}$, $k \equiv 3 \mod{10}$ and $k < 115$,
then we have $k = 83$ because $k \ge 45$.
Then by Lemma \ref{lem:7}, we have $f \equiv 0 \mod{p}$.

\noindent
(4) Finally, we assume $k \equiv 11 \mod{12}$ and $k \ge 115$.
To prove this case, we start with proving the following lemma.
\begin{Lem}
  \label{lem:9}
  Let $f = X_{35}g \in M_{k}(\gamt)_{\zkp}$ with
  $W(g) \equiv 0 \mod{p}$.
  Assume $k \equiv 11 \mod{12}$, $k \ge 115$ and \eqref{k4}.
  Then we have $W''(g) \equiv 0 \mod{p}$.
\end{Lem}
\begin{proof}
  We show the statement only for $p=2$. The case $p=3$ also can be proved by a
  similar argument.
  By Corollary \ref{cor:3},
  there exists $g'\in M_{k-43}(\Gamma _2)_{\mathbb{Z}_{(p)}}$ such that
  $g\equiv X_{10}g' \mod{p}$.
  Then, it follows from Lemma \ref{lem:10} that
  \[fX_{35}\equiv X_{10}X_{35}^2g'
    \equiv g'X_{10}^3(Y_{12}^2X_{16}^2+X_{10}^4) \mod{p}. \]
  By Lemma \ref{lem:3} and the assumption (\ref{k4}), we have
  \begin{align*}
    b_{k}+2 < v_p(\widetilde{f})+2\le
    v_p(\wf \wx_{35}) =v_{p}(\widetilde{g}' \wx_{10}\wx_{35}^2)=
    3+v_p(\widetilde{g}'(\wy_{12}^2\wx_{16}^2+\wx_{10}^4)).
  \end{align*}
  This implies that
  \begin{equation*}
    v_p(\widetilde{g}'(\wy_{12}^2\wx_{16}^2+\wx_{10}^4))> [(k-15)/10].
  \end{equation*}
  On the other hand, we have
  \begin{align*}
    W(g'(Y_{12}^2X_{16}^2+X_{10}^4))=W(g'Y_{12}^2X_{16}^2)
    \equiv W(g') \cdot y_{12}^2 \cdot x_{12}^2 \mod{p},
  \end{align*}
  where we used (\ref{eq:6}) and the fact $x_{4}\equiv 1$ mod $p$. By this
  congruence, $W(\widetilde{g}'(\wy_{12}^2\wx_{16}^2+\wx_{10}^4))$ can be
  regarded as of weight $k-43+48=k+5$. By $k\ge 115$, we have
  \[v_p(\widetilde{g}'(\wy_{12}^2\wx_{16}^2+\wx_{10}^4))> [(k-15)/10] \ge [(k+5)/12]. \]
  Applying Lemma \ref{lem:12}, we have
  \[W(g'(Y_{12}^2X_{16}^2+X_{10}^4))\equiv W(g') \cdot y_{12}^2 \cdot x_{12}^2\equiv 0 \mod{p}. \]
  This implies that
  \[W''(g)\equiv W(g')\cdot x_{12} \equiv 0 \mod{p}. \]
This completes the proof of the lemma.
\end{proof}
We shall return to proof of the case (4).
Since $W(g)\equiv 0 \mod{p}$ and $W''(g) \equiv 0 \mod{p}$,
there exists $h\in M_{k-55}(\Gamma _2)_{\mathbb{Z}_{(p)}}$ such that
$g\equiv X_{10}^2h \mod{p}$ by Proposition \ref{prop:1}.
Note that $f\equiv X_{10}^2X_{35}h \mod{p}$. If we put $f':=X_{35}h\in M_{k-20}(\Gamma _2)_{\mathbb{Z}_{(p)}}$, then
\[v_{p}(\wf)=v_p(\wx_{10}^2\wf')=2+v_p (\wf')> b_{k}. \]
This means that \[v_p(\wf')> b_{k-20}. \]
By the induction hypothesis, we get $f'\equiv 0 \mod{p}$.
Therefore we have $f \equiv 0 \mod{p}$.
This completes the proof.
\qed

\subsection{Proof of Corollary \ref{cor:1}}
As explained in the beginning of this section, we may assume
$\mathcal{O}_{\mathfrak{p}} = \zkp$, where $p$ is a prime number.
Let $\Gamma \subset \gamt$ be a congruence subgroup of level $N$
and $f \in M_{k}(\Gamma)_{\zkp}$.
By the proof of \cite[Theorem 1.3]{choi2013sturm}, there exists
$g \in M_{k (i - 1)}(\Gamma)_{\zkp}$ such that
\begin{equation*}
  f g \in M_{ki}(\gamt)_{\zkp}, \text{ and } g \not \equiv 0 \mod{p}.
\end{equation*}
Here $i = [\gamt: \Gamma]$.
We assume $v_{p}(\wf) > b_{ki}$.
Then by \eqref{eq:11}, we have
\begin{equation*}
  v_{p}(\wf \wg) \ge v_{p}(\wf) > b_{ki}.
\end{equation*}
By Theorem \ref{thm:1}, we have $\wf \wg = 0$.
Since $\wg \ne 0$, we have $\wf = 0$, i.e., $f \equiv 0 \mod{p}$.
This completes the proof.
\qed

\subsection{Proof of the sharpness}
We prove Theorem \ref{thm:2}.
If $k$ is even, then we can show the assertion of the theorem
by a similar argument to \cite{choi2013sturm}.
Let $k$ be odd and $f_{k}$ for $k = 35, 39, 41, 43$ and $47$
be modular forms given in Lemma \ref{lem:4}.
Then by Lemma \ref{lem:4}, we have
$\ldt(f_{k}X_{10}^{i}) = q_{1}^{2 + i}q_{2}^{-1 - i}q_{3}^{3 + i}$
for $k = 35, 39, 41, 43$ and
$\ldt(f_{47}X_{10}^{i})=q_{1}^{3 + i}q_{2}^{-2 - i}q_{3}^{4 + i}$.
Thus we have the assertion of the theorem.\qed

\section*{Acknowledgment}
The first named author is supported by JSPS Grant-in-Aid for Young Scientists (B) 26800026.
The second named author is partially supported by JSPS Kakenhi 23224001.

\providecommand{\bysame}{\leavevmode\hbox to3em{\hrulefill}\thinspace}
\providecommand{\MR}{\relax\ifhmode\unskip\space\fi MR }
\providecommand{\MRhref}[2]{%
  \href{http://www.ams.org/mathscinet-getitem?mr=#1}{#2}
}
\providecommand{\href}[2]{#2}

\begin{flushleft}
  Toshiyuki Kikuta\\
  Faculty of Information Engineering\\
  Department of Information and Systems Engineering\\
  Fukuoka Institute of Technology\\
  3-30-1 Wajiro-higashi, Higashi-ku, Fukuoka 811-0295, Japan\\
  E-mail: kikuta@fit.ac.jp
\end{flushleft}

\begin{flushleft}
  Sho Takemori\\
  Department of Mathematics, Hokkaido University\\
  Kita 10, Nishi 8, Kita-Ku, Sapporo, 060-0810, Japan \\
  E-mail: takemori@math.sci.hokudai.ac.jp
\end{flushleft}

\end{document}